\documentclass[a4paper,11pt]{article}

\usepackage[utf8]{inputenc}
\usepackage[T1]{fontenc}

\usepackage{amsmath}
\usepackage{amssymb}
\usepackage{bm}

\usepackage{amsfonts}
\usepackage{amsthm}

\usepackage{titlesec}
\titleformat{\subsection}[runin]{\normalfont\large\bfseries}{\thesubsection}{6pt}{}[]
\titleformat{\subsubsection}[runin]{\normalfont\bfseries}{\thesubsubsection}{5pt}{}[]

\usepackage[ruled,vlined,linesnumbered]{algorithm2e}

\usepackage{graphicx}
\graphicspath{{BalancedImages/}}

\usepackage{wrapfig}
\usepackage{lipsum}
\usepackage{multirow}
\usepackage{enumitem}
\usepackage{longtable}
\usepackage[capitalise]{cleveref}

\usepackage{caption}
\usepackage{subcaption}
\usepackage{float}

\title{\textbf{\Huge{Classification of balanced toral elements of exceptional Lie algebras}}}
\author{Floriana Amicone}
\date{\vspace{-5ex}}

\theoremstyle{plain}
\newtheorem{teo}{Theorem}[section]

\newtheorem{prop}[teo]{Proposition}
\newtheorem{cor}[teo]{Corollary}

\theoremstyle{definition}

\begin{document}

\maketitle

\paragraph{Abstract:} Let $\mathfrak{g}$ be a simple Lie algebra of exceptional type over an algebraically closed field $k$, and let $G$ be a simple linear algebraic group with Lie algebra $\mathfrak{g}$. For $\mathrm{char} \, k =p >0$, we present a
complete classification of the $G$-conjugacy classes of balanced toral elements of $\mathfrak{g}$. As a result, we also obtain the classification of conjugacy classes of balanced inner torsion automorphisms of $\mathfrak{g}$ of order $p$ when $\mathrm{char} \, k =0$.

\section{Introduction and statement of results}

Let $\mathfrak{g}$ be the Lie algebra of a connected algebraic group $G$, defined over a field $k = \overline{k}$ of positive characteristic $p$.  

For all $x \in \mathfrak{g}$ the derivation $(\mathrm{ad} \, x)^p$ is still inner and equals $\mathrm{ad} \, x^{[p]}$ for a certain $ x^{[p]} \in \mathfrak{g}$. Therefore, the Lie algebra $\mathfrak{g}$ is canonically endowed with a $G$-equivariant $p$-semilinear $[p]$-th power map:
\begin{displaymath}
x \in \mathfrak{g} \longmapsto x^{[p]} \in \mathfrak{g}.
\end{displaymath}

An element $h \in \mathfrak{g}$ is said to be \textit{toral} if it verifies $h^{[p]}=h$. Such an element is semisimple and acts on $\mathfrak{g}$ with eigenvalues belonging to the finite field $\mathbb{F}_p \subseteq k$, because the minimal polynomial of $\mathrm{ad} \, h$ divides the polynomial $X^P-X$. We will note  $\mathfrak{g}(h,i)=\{ x \in \mathfrak{g} \, | \, [h,x]=ix \}$ the eigenspace relative to the eigenvalue $i \in \mathbb{F}_p$.

A toral element $h$ is called \textit{balanced} if the eigenspaces $\mathfrak{g}(h,i)$ corresponding to nonzero eigenvalues are all of the same dimension. If in addition their dimension is divisible by a certain $d \in \mathbb{N}$, the element is called \textit{$d$-balanced} (see section \ref{Preli} for explicit notation).

\subsection{} 
Premet (\cite{ModularMorozov}) gave a complete account on $p$-balanced toral elements of exceptional Lie algebras, where $p$ is a good prime for the root system of $\mathfrak{g}$. He observed that there is a natural way of associating a nilpotent orbit (more specifically, a Richardson orbit) to a $p$-balanced toral element $h$. Indeed, this can be done by inducing from the $0$-orbit on the centralizer of $h$ in $\mathfrak{g}$, which is a Levi subalgebra of $\mathfrak{g}$. This orbit enjoys some particular properties, notably its dimension is divisible by $p(p-1)$ and its elements belong to the restricted nullcone $\mathcal{N}_p(\mathfrak{g})$. Hence, one can resort to case-by-case computations involving Richardson orbits of this type, since they are not too many.

The purpose of this work is to classify all balanced toral elements of a simple Lie algebra $\mathfrak{g}$ of exceptional type in positive characteristic $p$. Therefore, less restrictive hypothesis will be taken: we allow $d$ to be any nonnegative integer and we drop the assumption of $p$ being a good prime for the root system of $\mathfrak{g}$. In characteristic $2$ and $3$ every toral element is automatically balanced, therefore the only nontrivial bad prime to consider is $p=5$ for the root system of type $E_8$; this will be examined after settling the good characteristic case. Some of the features used in \cite{ModularMorozov} can be adapted to this more general setting. For instance, one can perform the same construction of a Richardson orbit from a conjugacy class of balanced toral elements (at least in good characteristic), but under our assumptions the dimension of the orbit has to be divisible only by $p-1$ (see section \ref{Class}).

\subsection{}
The description of conjugacy classes of balanced elements will be given in terms of Kac coordinates. These are a sequence of nonnegative integers introduced by Kac (\cite[Chapter 8]{Kac-InfLieAlg}) in order to classify finite order (or \textit{torsion}) automorphisms of a simple Lie algebra in characteristic $0$. This draws a connection between toral elements in characteristic $p>0$ and torsion automorphisms of order $p$ of a simple Lie algebra defined over $\mathbb{C}$, so that results valid for the former admit natural analogues holding for the latter.

Assuming $k = \mathbb{C}$, let $\mathfrak{g}_\mathbb{C}$ be a simple Lie algebra. The identity component of its automorphism group $(Aut \, \mathfrak{g}_\mathbb{C})^\circ = G_\mathbb{C}$ is an algebraic group of adjoint type with $\mathrm{Lie} \, G_\mathbb{C} = \mathfrak{g}_\mathbb{C}$, while the other irreducible components are laterals of the form $\varphi G_\mathbb{C}$, where $\varphi$ is a graph automorphism of the Dynkin diagram. Torsion automorphisms of $\mathfrak{g}_\mathbb{C}$ of order $m$ are in bijection with $\mathbb{Z}/ m \mathbb{Z}$-gradings on $\mathfrak{g}_\mathbb{C}$. Moreover, a finite order automorphism is inner (that is, it belongs to $G_\mathbb{C}$) if and only if a maximal toral subalgebra of $\mathfrak{g}_\mathbb{C}$ is pointwise fixed.

Our case of interest is the grading of $\mathfrak{g}$ induced by the adjoint action of a toral element $h$. This is a $\mathbb{Z}/ p \mathbb{Z}$-grading in characteristic $p$, for which any Cartan subalgebra containing $h$ is included in the $0$-degree component. Although automorphisms of order $m$ are in bijection with $\mathbb{Z}_m$-gradings on $\mathfrak{g}$ only when $m$ and $p$ are relatively prime, at least in the inner case Kac coordinates can still be used over fields of positive characteristic to parameterize embeddings of the affine group scheme $\bm{\mu}_p$ in $G$ (\cite{Serre-coo}). It turns out that, in characteristic $p>0$, Kac coordinates relative to such embeddings of $\bm{\mu}_p$ parameterize conjugacy classes of toral elements of $\mathfrak{g}$.

\subsection{} Section \ref{Class} contains the description of the algorithm. Each balanced toral element yields a Richardson orbit of dimension divisible by $p-1$; since there are only finitely many of these orbits, it is possible to check them one by one. Roughly speaking, we seek tuples of Kac coordinates that provide a balanced element. In many instances one can rule out straightaway certain choices of Kac coordinates using combinatorial arguments. Unfortunately this approach is unsuccessful or too involved when both the dimension of the nilpotent orbit and the characteristic are big. In order to tackle all the possible cases, we implemented the algorithm on the coding language C (\cite{GNUManual}).

As remarked, manual computations are doable and effective in almost all cases. Before realising the trickiest cases could be dealt with by using a computer, we had proven that for exceptional Lie algebras there is a natural injection from the set of connected subgraphs of the extended Dynkin diagram to the set of positive roots $\Phi^+$. A proof of this can be found in section \ref{FAlternative}, along with examples of how we used it.

\subsection{} It emerges from the Tables in section \ref{Tables} that sometimes there exist more than one conjugacy class of balanced toral elements in char $p$ having isomorphic (and $G$-conjugate) centralizers. Moreover, if $h$ is balanced, $rh$ is balanced for any $r \in \mathbb{F}_p^\times$. One may ask whether $rh$ is still conjugate to $h$. Combining the arguments in sections \ref{Class} and \ref{ScalarMultiples} we prove the following:

\begin{teo} \label{ScalarM}
Let $G$ be a connected reductive algebraic group of exceptional type over a field $k$ of characteristic $p>0$ and let $\mathfrak{g}=\mathrm{Lie} \, G$.
\begin{itemize}
\item[$(a)$] The Kac coordinates of $G$-orbits of balanced toral elements are listed in \cref{G2,F4,E6,E7,E8}.
\item[$(b)$] Let $h, h^\prime \in \mathfrak{g}$ be balanced toral elements with isomorphic centralizers. There exists $r \in \mathbb{F}_p^\times$ such that $rh$ is conjugate to $h^\prime$ under an automorphism of $\mathfrak{g}$.
\end{itemize}
\end{teo}

\noindent An equivalence relation that extends $G$-conjugacy can be defined on balanced toral elements by saying that $h$ and $h^\prime$ are in relation if there exists $r \in \mathbb{F}_p^\times$ such that $rh$ is conjugate to $h^\prime$ under an element of $Aut(\mathfrak{g})$. Theorem \ref{ScalarM} states that in chararacteristic $p$, once the type of centralizer is fixed, there exists at most one equivalence class for this relation.

\subsection{} These results have an analogous interpretation in the setting of torsion automorphisms of order $p$ of a complex Lie algebra $\mathfrak{g}_\mathbb{C}$.

Let $\sigma \in Aut(\mathfrak{g})$ be an element of order $p >0$, so that the Lie algebra decomposes as $\mathfrak{g}= \bigoplus_{i \in \mathbb{F}_p} \mathfrak{g}(\sigma, i)$, where $\mathfrak{g}(\sigma, i)= \{ x \in \mathfrak{g} \, | \, \sigma(x)= \xi^i x \}$, for $\xi \in \mathbb{C}$ a fixed primitive $p$-th root of $1$. $\sigma$ is said to be balanced if $\dim \mathfrak{g}(\sigma, i) = \dim \mathfrak{g}(\sigma, 1)$ for all $i \in \mathbb{F}_p^\times$. The analogue of Theorem \ref{ScalarM} holds:

\begin{teo}
Let $\mathfrak{g}_\mathbb{C}$ a simple Lie algebra of exceptional type over $\mathbb{C}$, $G_\mathbb{C}$ a simple linear algebraic group with Lie algebra $\mathfrak{g}_\mathbb{C}$ and $Aut (\mathfrak{g}_\mathbb{C})$ the automorphisms group of $\mathfrak{g}_\mathbb{C}$.
\begin{itemize}
\item[$(a)$] The Kac coordinates of $G_\mathbb{C}$-orbits of balanced automorphisms of order $p>0$, for $p \in \mathbb{N}$ prime, are listed in \cref{G2,F4,E6,E7,E8}.
\item[$(b)$] Let $\sigma, \sigma^\prime \in \mathfrak{g}$ be balanced torsion automorphisms with $G_\mathbb{C}$-conjugate fixed points subalgebras. There exists $r \in \mathbb{F}_p^\times$ such that $\sigma^r$ is conjugate to $\sigma^\prime$ under an element of $Aut(\mathfrak{g}_\mathbb{C})$.
\end{itemize}
\end{teo}

\paragraph{Acknowledgements} I am extremely grateful to Alexander Premet for the guidance and advice received while working at this paper, and for giving me precious suggestions on how to improve a preliminary version.

\section{Notation and Preliminaries} \label{Preli}

\subsection{General setting.} 
Let $\mathfrak{g}$ be the Lie algebra of a connected reductive algebraic group $G$, defined over an algebraically closed field $k$ of characteristic $p>0$; $G$ acts on its Lie algebra via the adjoint action. Unless otherwise specified, $\mathfrak{g}$ and $G$ will be of exceptional type and will satisfy the \textit{standard hypothesis}:
\begin{itemize}
\item $p$ is a good prime for $G$;
\item the derived subgroup of $G$ is simply connected;
\item $\mathfrak{g}$ admits a nondegenerate $G$-equivariant symmetric bilinear form $\kappa(\cdot,\cdot): \mathfrak{g} \times \mathfrak{g} \rightarrow k$.
\end{itemize}

\noindent Let $\Phi$ be the root system of $G$ and $\Delta=\{ \alpha_1, \ldots, \alpha_l \}$ a base of simple roots.  If $\widetilde{\alpha}_0 = \sum_{i=1}^{l} b_i \alpha_i$ is the highest root of $\Phi$, we say that $p$ is a \textit{good prime} for $G$ if it does not divide any of the integers $b_i$, for $i=1, \ldots, l$. When $G$ is of exceptional type, the standard hypothesis are satisfied if and only if $p$ is a good prime for $G$. If $p$ is a good prime for $G$, it is a good prime for any Levi subgroup of $G$.

In what follows $\mathfrak{g}_\alpha \subseteq \mathfrak{g}$ will denote the $1$-dimensional root subspace for a fixed root $\alpha \in \Phi$. Whenever we want to refer to a simple exceptional algebraic group and its Lie algebra defined over $\mathbb{C}$ (but also other notions, e.g. orbits), we will simply add a subscript, as in $\mathfrak{g}_\mathbb{C}$ and $G_\mathbb{C}$.

Retain notations from the introduction. We saw that a \textit{toral} element $h=h^{[p]}$ decomposes $\mathfrak{g}$ as a direct sum $\mathfrak{g}=\mathfrak{g}(h,i)$, where $\mathfrak{g}(h,i)$ is the eigenspace of eigenvalue $i \in \mathbb{F}_p$ for $\mathrm{ad} \, h$.

Recall that a toral element $h$ is said to be \textit{balanced} if the dimension of $\mathfrak{g}(h,i)$ for $i \neq 0$ is independent of $i$; more explicitly $\dim \mathfrak{g}(h,i) = \dim \mathfrak{g}(h,1)$ for all $i \in \mathbb{F}_p^\times$. If a positive integer $d$ divides $\dim \mathfrak{g}(h,i)$, then $h$ is called \textit{$d$-balanced}.
% if in addition their dimension is divisible by a certain $d \in \mathbb{N}$, the element is called \textit{$d$-balanced}.

\subsection{Induced nilpotent orbits} \hfill\break

Let $\mathfrak{p}$ be a parabolic subalgebra of $\mathfrak{g}$ that admits a Levi decomposition $\mathfrak{p}=\mathfrak{l} \oplus \mathfrak{n}$, where $\mathfrak{l}$ is a Levi subalgebra and $\mathfrak{n}$ is the nilradical of $\mathfrak{p}$. Let $\mathcal{O}_\mathfrak{l}$ be a nilpotent orbit in $\mathfrak{l}$. There is a unique nilpotent $G$-orbit $\mathcal{O}_\mathfrak{g}$ in $\mathfrak{g}$ meeting $\mathcal{O}_\mathfrak{l}+\mathfrak{n}$ in an open dense subset; this is because the Lie algebra $\mathfrak{g}$ admits only a finite number of nilpotent $G$-orbits in any characteristic (see \cite{JantzenNilp}, for example). The orbit $\mathcal{O}_\mathfrak{g}$ is \textit{induced} by $\mathcal{O}_\mathfrak{l}$ in the sense of \textit{Lusztig-Spaltenstein}. If $\mathcal{O}_\mathfrak{l}=\{ 0 \}$, $\mathcal{O}_\mathfrak{g}$ is called a \textit{Richardson orbit}.

The orbit induced by $\mathcal{O}_\mathfrak{l}$ depends only on the Levi subalgebra $\mathfrak{l}$,  not on the choice of a parabolic subalgebra containing it; for this reason it will be denoted $\mathrm{Ind}^\mathfrak{g}_\mathfrak{l}(\mathcal{O}_\mathfrak{l})$.
There exist nilpotent orbits that can be induced in more than one way; an orbit which is not induced is called \textit{rigid} (\cite{LusztSpalt}).

\subsection{Sheets} \hfill\break

Let $m \in \mathbb{Z}$ be fixed, and let $\mathfrak{g}_{(m)}$ be the set of elements in $\mathfrak{g}$ whose centralizer in $G$ has dimension $m$. The irreducible components of the $\mathfrak{g}_{(m)}$ are called the \textit{sheets} of $\mathfrak{g}$.
In other words, the sheets of $\mathfrak{g}$ are the irreducible components of each of the varieties consisting of the union of all $G$-orbits of the same dimension.

\subsubsection{Decomposition classes.} \label{rigiddec}

A simple Lie algebra admits only a \textit{finite number} of sheets. Moreover, sheets can be parameterized by $G$-conjugacy classes of pairs $(\mathfrak{l}, (\mathrm{Ad}L)e)$, where $\mathfrak{l}$ is a Levi subalgebra of $\mathfrak{g}$ and $(\mathrm{Ad} L)e$ (or, more briefly, $L e$) is a \textit{rigid} nilpotent orbit on $\mathfrak{l}$. These results for groups satisfying the standard hypothesis follow from \cite{RigidSheets}, we briefly summarize how the correspondence works as it will be useful for building the theory in later sections.

\vspace{2mm}
Let $\mathfrak{l}$ be a Levi subalgebra of $\mathfrak{g}$, $\mathfrak{z(l)}_{reg}$ be the set of all elements $x$ in the centre of $\mathfrak{l}$ such that $\mathfrak{l}=\mathfrak{c_g}(x)$ and let a nilpotent element $e_0 \in \mathfrak{l}$ be fixed. The \textit{decomposition class} of
$\mathfrak{g}$ associated with the pair $(\mathfrak{l}, e_0)$ is defined as:

\begin{displaymath}
\mathcal{D}(\mathfrak{l}, e_0)= (Ad \, G) \cdot (e_0 + \mathfrak{z(l)}_{reg}).
\end{displaymath}

\noindent There exist finitely many decomposition classes and each of them is included in one of the sheets, therefore every sheet $\mathcal{S}$ contains a unique open decomposition class.

\vspace{2mm}
Nonetheless, \textit{every sheet contains a unique nilpotent orbit}. Again, this holds for groups satisfying the standard hypothesis and it has been proven in \cite[Proposition 2.5]{RigidSheets}. More generally, the closure of every decomposition class $\mathcal{D}(\mathfrak{l},e_0) $ contains a unique \textit{open} nilpotent orbit $\mathcal{O}(e)$, that intersects densely with $L e_0 + \mathfrak{n}_+$ (\cite[Theorem 2.3 and Proposition 2.5]{RigidSheets}). This means in particular that $\mathcal{O}(e)$ is \textit{the orbit obtained from $L e_0$ by Lusztig-Spaltenstein induction}: if $\mathcal{D}(\mathfrak{l},e_0)$ is the open decomposition class of $\mathcal{S}$, the unique nilpotent orbit in $\mathcal{S}$ is $\mathrm{Ind}_{\mathfrak{l}}^{\mathfrak{g}} (L e_0)$.

\subsection{Sheet diagrams and Richardson orbits}

\subsubsection{Weighted Dynkin diagrams.}
It is well-known that over $\mathbb{C}$ (and in characteristic $p \gg 0$), a nilpotent orbit $\mathcal{O}_\mathbb{C}$ is uniquely determined by its \textit{weighted Dynkin diagram}, a way of labelling the Dynkin diagram described hereafter. An element $e \in \mathcal{O}_\mathbb{C}$ is included in an \mbox{$\mathfrak{sl}_2$-triple} $(h,e,f)$ of $\mathfrak{g}_\mathbb{C}$, uniquely determined up to $G_\mathbb{C}$-conjugacy by the $G_\mathbb{C}$-orbit of $h$ (and that of $e$).

There exist a maximal toral subalgebra $\mathfrak{t}_\mathbb{C}$ containing $h$ and a system of simple roots $\Delta$ such that $\alpha(h) \geq 0$ for all $\alpha \in \Delta$. There are not many possibilities for the value assumed by $\alpha(h)$ for $\alpha \in \Delta$:
it can only be equal to $0,1$ or $2$ (see \cite{CartFgLt}, for example).

Let $\Delta=\{\alpha_1, \ldots, \alpha_l\}$; we label the node corresponding to $\alpha_i$ in the Dynkin diagram of $\Phi$ with the integer $\alpha_i(h)$. This is the \textit{weighted Dynkin diagram} of the nilpotent orbit $\mathcal{O}_\mathbb{C}$.

\paragraph{} In the case where $\mathrm{char} \, k$ is a good prime for the root system of $\mathfrak{g}$, nilpotent orbits are still parameterized by the same weighted Dynkin diagrams (see \cite{CartFgLt}). Yet, they must be interpreted in a different way, since the theory of $\mathfrak{sl}_2$-triples fails when $\mathrm{char} \, k$ is not sufficiently big: see \cite[Section 2]{KempfRouss-Theo}, for example.

Let $e \in \mathcal{N}(\mathfrak{g})$, the nilpotent cone of $\mathfrak{g}$; there exists a one-parameter subgroup $\lambda$ optimal for $e$ in the sense of the Kempf-Rousseau theory, so that $e \in \mathfrak{g}(\lambda,2)$, where $\mathfrak{g}(\lambda,i)= \{ x \in \mathfrak{g} | (Ad \, \lambda(t)) x = t^i x \}$.

Up to replacing $e$ and $\lambda$ by conjugates, there exists a system of simple roots $\Delta=\{\alpha_1, \ldots, \alpha_l\}$ such that $\forall \; x \in \mathfrak{g}_{\alpha_i}$ one has $(Ad \, \lambda(t)) x = t^{r_i} x$ with $r_i \in \{0,1,2 \}$, for $i= 1, \ldots, l$. The Dynkin diagram labelled with the $r_i$, let us note it $D$, coincides with one of the weighted Dynkin diagrams relative to an orbit over $\mathbb{C}$. Moreover, any weighted Dynkin diagram over $\mathbb{C}$ can be obtained in this way, and $\dim_k \mathcal{O}(D) = \dim_\mathbb{C} \mathcal{O}_\mathbb{C}(D)$.

\subsubsection{Sheet diagrams.}
As remarked in \ref{rigiddec}, sheets are parameterized by $G$-orbits of pairs $(\mathfrak{l}, e_0)$, where $(\mathrm{Ad} L) e_0$ is rigid in $\mathfrak{l}$. \cite{InducedExceptionalOrbits} contains a list of rigid nilpotent orbits of exceptional Lie algebras in characteristic $0$ and their Dynkin diagrams (Tables $1$-$5$). Such information can be exploited in our setting as well: for $\mathfrak{g}$ of exceptional type, rigid orbits in good characteristic admit the same weighted Dynkin diagram as their counterparts over $\mathbb{C}$ (\cite[Theorem 3.8]{RigidSheets}).

\paragraph{} Let $\mathcal{S}$ be the sheet associated to the $G$-orbit of pairs $(\mathfrak{l}, e_0)$. Up to conjugacy, we may assume that $\mathfrak{l}$ is a standard Levi subalgebra of a standard parabolic subalgebra $\mathfrak{p}=\mathfrak{p}_\Pi$ corresponding to a subset $\Pi \subseteq \Delta$. The \textit{sheet diagram} (\cite{InducedExceptionalOrbits}) is a labelling of the Dynkin diagram obtained by applying the procedure hereafter. Nodes corresponding to roots in $\Delta \setminus \Pi$ are labelled $2$. As $\Pi$ is a system of simple roots for the root system of $\mathfrak{l}$, the nodes corresponding to roots in $\Pi$ are labelled just as in the weighted Dynkin diagram of the $L$-orbit of $e$.

A nilpotent orbit is Richardson if and only if the sheet diagram does not contain any node labelled $1$ \cite[5.1]{InducedExceptionalOrbits}.

 Two standard Levi subalgebras can be $G$-conjugated, and therefore the sheet diagram of a Richardson orbit is determined up to conjugation of the standard Levi subalgebra corresponding to the choice of simple roots with label $0$. Nevertheless, a certain orbit can have more than one sheet diagram in the case it belongs to more than one sheet.

\subsubsection{} In order to achieve the classification of balanced toral elements, we will look in detail at Richardson orbits of exceptional Lie algebras, and in particular at the structure of the Levi subalgebra from which they are induced. This latter describes the centralizer of the balanced toral element we are looking for (\ref{List}), which can be recovered from the sheet diagram of the orbit. We will make use of \cite[Tables 6-10]{InducedExceptionalOrbits}, which provide information on the dimension and the sheet diagram of every induced orbit. Notice that data provided \emph{loc. cit.} can be used more generally in good characteristic thanks to \cite[Theorem 1.4]{RigidSheets}.

\subsection{Kac coodinates and finite order automorphisms of simple Lie algebras} \label{InnerTorsion}

\subsubsection{Inner torsion automorphisms.}

Until further notice, we will be working with a simple Lie algebra $\mathfrak{g}_\mathbb{C}$ and the algebraic group of adjoint type $G_\mathbb{C}= Aut(\mathfrak{g}_\mathbb{C})^\circ$. This section follows closely \cite{ReedTors}.

Let us fix a Borel subgroup $B_\mathbb{C} \subseteq G_\mathbb{C}$ and a maximal torus $T_\mathbb{C} \subseteq B_\mathbb{C}$; this amounts to the choice of a system of positive roots $\Phi^+$ for the root system $\Phi$ of $G_\mathbb{C}$, and hence a subset of simple roots $\Delta = \{ \alpha_1, \ldots , \alpha_l\}$. $\Delta$ is a basis of the weight lattice $X = X^\ast (T_\mathbb{C})= \{ \mathrm{algebraic \; homomorphisms} \; T_\mathbb{C} \rightarrow \mathbb{C}^\times \}$, that admits a natural pairing $\langle \,,\rangle : X \times Y \rightarrow \mathbb{Z}$, where $Y= X_\ast (T_\mathbb{C})= \{ \mathrm{algebraic \; homomorphisms} \; \mathbb{C}^\times \rightarrow T_\mathbb{C} \}$ is the lattice of cocharacters of $T_\mathbb{C}$. Let $\{ \check{\omega}_{1}, \ldots, \check{\omega}_{l}\}$ be the $\mathbb{Z}$-basis of $Y$ consisting of fundamental coweights dual to $\Delta$, explicitly $\langle \alpha_i, \check{\omega}_j \rangle = \delta_{ij}$ for all $i,j$.

\subsubsection{}
The full automorphism group $Aut(\mathfrak{g}_\mathbb{C})$ consists of laterals $\sigma G_\mathbb{C}$, where $\sigma$ is a graph automorphism of the Dynkin diagram. An automorphism of $\mathfrak{g}_\mathbb{C}$ belonging to $G_\mathbb{C}$ is called \textit{inner}. An inner torsion automorphism is semisimple, so we can assume it belongs to $T_\mathbb{C}$. If it is of order $m \in \mathbb{Z}_{>0}$, it is uniquely determined by its action on root subspaces $(\mathfrak{g}_\mathbb{C})_{\alpha_i }$, that is, multiplication by an $m$-th root of unity. Noted $V = \mathbb{R} \otimes Y$, torsion elements are the image of $\mathbb{Q} \otimes Y$ under the homomorphism
$$
\mathrm{exp}: V \longrightarrow T_\mathbb{C}
$$
defined by
$$\alpha(\mathrm{exp}(x)) = e^{2 \pi i \langle \alpha, x \rangle}$$
for all $\alpha \in \Phi$, where the pairing is extended to $V$ by $\mathbb{R}$-bilinearity. The homomorphism $\mathrm{exp}$ yields an isomorphism between $V/Y$ and the subtorus $\mathrm{Im} (\mathrm{exp}) \subseteq T_\mathbb{C}$.

The Weyl group $W$ acts on $V$ as a group generated by reflections about simple roots; every semisimple element of $G_\mathbb{C}$ is conjugate to an element of $T_\mathbb{C}$, and two elements of $T_\mathbb{C}$ are $G_\mathbb{C}$-conjugate if and only if they are $W$-conjugate. Therefore, $G_\mathbb{C}$-conjugacy classes of semisimple elements of $G_\mathbb{C}$ are in bijection with $W$-conjugacy classes on $T_\mathbb{C}$, so that two elements $x, y \in \mathbb{Q}\otimes Y$ give conjugate torsion elements $\mathrm{exp}(x)$ and $\mathrm{exp}(y)$ if and only if $x, y$ are conjugate in $V$ under the \textit{extended affine Weyl group}
\begin{displaymath}
\widetilde{W} = W \ltimes Y,
\end{displaymath}
where the action of $Y$ on $V$ is given by translations.

\subsubsection{}
The hyperplane arrangement on $V$ consisting of hyperplanes
$$
L_{\alpha, n} = \{  x \in V \; | \; \langle \alpha, x \rangle = n\} \qquad  \alpha \in \Phi, n \in \mathbb{Z}
$$
 has been extensively studied; we refer in particular to \cite[VI.2]{ErBourba}. The group of isometries generated by reflections about the hyperplanes $L_{\alpha, n}$ naturally arises in this setting, and it is isomorphic to the \textit{affine Weyl group} $\widetilde{W}^\circ = W \ltimes \mathbb{Z} \check{\Phi}$, where $\mathbb{Z} \check{\Phi} \subseteq Y$ is the lattice of coroots of $T_\mathbb{C}$. The space $V$ can be partitioned into \textit{alcoves}; the closure of any alcove is a fundamental domain for the action of $\widetilde{W}^\circ$, which is a Coxeter group generated by reflections about the $l+1$ \textit{walls} of any alcove.

Let $\widetilde{\alpha}_0 = \sum_{i=1}^{l} b_i \alpha_i$ be the highest root in $\Phi^+$, $\alpha_0 = 1-\widetilde{\alpha}_0$ (as an affine linear function on $V$), and $b_0=1$, so that $\sum_{i=0}^{l} b_i \alpha_i =1$.
The simplex:
\begin{displaymath}
C = \{ x \in V \; | \; \langle \alpha_i,x \rangle > 0 \quad \forall \, 0 \leq i \leq l \; \}
\end{displaymath}
is the alcove determined by $\Delta$.

For sake of notation, let $\check{\omega}_0=0$. The closure of $C$ can be written adding an extra coordinate $x_0$ as
\begin{displaymath}
\overline{C} = \left\{ \, \sum_{i=0}^{l} x_i \check{\omega}_i \; \Big| \; x_i \geq 0 \, \mbox{and} \; \sum_{i=0}^{l} b_i x_i =1 \, \right\}.
\end{displaymath}

\noindent Notice that $\overline{C}$ is the convex hull of its vertices defined as $v_i=b_i^{-1} \check{\omega}_i$ for $0 \leq i \leq l$.

\vspace{2mm}
The extended affine Weyl group $\widetilde{W}$ acts transitively on alcoves since the affine Weyl group $\widetilde{W}^\circ$ does. Yet, unlike this latter, $\widetilde{W}$ need not act simply transitivity on alcoves, so there exists a possibly nontrivial stabilizer $\Omega= \{  \rho \in \widetilde{W} \, | \, \rho \cdot C = C \}$, which gives the extended affine Weyl group a semidirect product structure as
\begin{displaymath}
\widetilde{W} = \Omega \ltimes \widetilde{W}^\circ.
\end{displaymath}

\noindent Hence for $x,y \in \overline{C}$ the elements $\mathrm{exp}(x)$ and $\mathrm{exp}(y)$ are \mbox{$G_\mathbb{C}$-conjugate} if and only if $x, y$ are $\Omega$-conjugate.

\subsubsection{Action of $\Omega$.}

 Looking at the definition of the extended and the unextended affine Weyl group one sees that $\Omega \simeq Y/\mathbb{Z} \check{\Phi}$, the fundamental group of $G_\mathbb{C}$. In our cases of interest $\Omega$ is the trivial group for types $E_8, F_4$ and $G_2$, while for types $E_6$ and $E_7$ it is isomorphic to a cyclic group of order $3$ and $2$ respectively.

 The action of $\Omega$ on the vertices of the alcove $C$ can be visualized via symmetries of the extended Dynkin diagram, as described in \cite[VI.2.3]{ErBourba}. Elements of $\Omega$ are in bijection with \textit{minuscule} coweights, namely those coweights $\check{\omega}_i$ for which $b_i=1$.

 In type $E_6$ the minuscule coweights are $\check{\omega}_1$ and $\check{\omega}_6$ in Bourbaki's notation, and the group $\Omega$ is isomorphic to the group of rotations of the extended Dynkin diagram.

 In type $E_7$ there is only one minuscule coweight, $\check{\omega}_7$, and the generator of $\Omega$ is the unique symmetry of the extended Dynkin diagram.

\subsubsection{Kac coordinates.}

Kac coordinates parameterize elements of $\overline{C}$ that correspond to finite order automorphisms. Let $x \in \overline{C}$ and suppose that $\mathrm{exp}(x)$ is a torsion element of order $m$, that is, $\mathrm{exp}(mx)=1$. Then $mx \in Y$, so there exist nonnegative integers $a_1, \ldots, a_l$ (with $\mathrm{gcd}\{ m, a_1, \ldots, a_l\}=1$) such that
\begin{displaymath}
x= \frac{1}{m} \sum_{i=1}^{l}a_i \check{\omega}_i.
\end{displaymath}
Since $x \in \overline{C}$,
\begin{displaymath}
0 \leq \langle \alpha_0,x \rangle = 1 - \frac{1}{m} \sum_{i=1}^{l} b_i a_i.
\end{displaymath}
If we define the nonnegative integers $b_0=1$ and $a_0 = m -\sum_{i=1}^{l} b_i a_i$, the $(l+1)$-tuple $(a_0, \ldots, a_l)$ satisfies
\begin{equation} \label{Kaccoo}
\sum_{i=0}^{l} b_i a_i = m.
\end{equation}
 We call the $l+1$-tuple of the $a_i$'s the \textit{Kac coordinates} of $s \in G_\mathbb{C}$. Fix a primitive $m$-th root of unity $\zeta$; Kac coordinates uniquely determine a torsion element $s$ since they give its action on every root subspace $(\mathfrak{g}_\mathbb{C})_{\alpha}$ for any root $\alpha = \sum_{i=1}^{l} c_i \alpha_i \in \Phi$, that is, multiplication by $\zeta^{\sum_{i=1}^{l} c_i a_i}$.

 Two elements $s = \mathrm{exp}(x)$ and $s^\prime=\mathrm{exp}(x^\prime)$ are $G_\mathbb{C}$-conjugate if and only if their Kac coordinates $(a_0, \ldots, a_l)$ and $(a^\prime_0, \ldots, a^\prime_l)$ can be obtained from one another under the action of $\Omega$ as permutation of the indices $\{0, \ldots , l \}$
 corresponding to its action on the vertices of the alcove $C$.

\subsection{Kac coordinates for toral elements}  \label{ToralKac} \hfill\break

From now on, $\mathfrak{g}$ and $G$ will be considered again as defined over a field $k$ of positive characteristic $p$.

Unfortunately the theory depicted so far does not always apply to this more general setting as in the characteristic $p>0$ case Kac coordinates can only be used to classify torsion automorphisms of order $n$ with $\mathrm{gcd}(p,n)=1$. Indeed, an automorphism of $\mathfrak{g}$ of order divisible by $p$ does not endow the Lie algebra with a $\mathbb{Z}/n \mathbb{Z}$-grading.

Serre (\cite{Serre-coo}) makes up for this issue by resorting to affine group schemes; for generalities about these we refer to \cite{IntroGroupSchemes}.

\subsubsection{}
Let us denote by $\bm{\mu}_n, \bm{G}_m$ and $\mathbb{Z}/n\mathbb{Z}$ the $k$-group schemes of $n$-th roots of unity, multiplicative group and integers modulo $n$ respectively. $\mathbb{Z}/n\mathbb{Z}$ is the Cartier dual of $\bm{\mu}_n$, in other words it is isomorphic to the character group $\mathrm{Hom}_k(\bm{\mu}_n, \bm{G}_m)$.

As for inner torsion automorphisms, Serre remarks that already in the characteristic $0$ case they are embedding of group schemes $\mathrm{Hom}_k(\bm{\mu}_n, G)$. Due to Cartier duality, they correspond to $\mathbb{Z}/n\mathbb{Z}$-gradings on the Lie algebra $\mathfrak{g}$, that can be parameterized by means of Kac coordinates.

This makes it possible to generalize the theory to the case of characteristic $p>0$, with $p|n$. In fact, Kac coordinates parameterize classes of embeddings of group schemes, or equivalently $\mathbb{Z}/n\mathbb{Z}$-gradings on $\mathfrak{g}$, and these correspond to elements $x \in \mathbb{Q} \otimes Y$ (here we have to replace the map $\mathrm{exp}$ with what Serre calls a \textit{parameterization of roots of unity}, that is, a homomorphism $e : \mathbb{Q} \rightarrow k$ of kernel $\mathbb{Z}$).

 The difference is that now an embedding of group schemes need not give an automorphism of order $n$, but it can be interpreted as a toral element $h \in \mathfrak{g}$ when $n=p$. Using arguments identical to those applied before, $G$-conjugacy classes on $\mathfrak{g}$ correspond to $\widetilde{W}$-orbits on $\mathbb{Q} \otimes Y$. A representative $h$ of a class of toral elements is uniquely defined by the Kac coordinates $(a_0, a_1, \ldots, a_l)$ of a point $x \in \mathbb{Q} \otimes Y \cap \overline{C}$, so that $\alpha_i(h) = a_i$ for all $i= 1, \ldots, l$, and $a_0=n -\sum_{i=1}^{l} b_i a_i$.

\section{Classification of balanced toral elements} \label{Class}

The machinery from section \ref{Preli} will be now used to obtain the classification of $G$-conjugacy classes of balanced toral elements. In characteristic $p=2$ and $p=3$ all toral elements are balanced. This is obvious for $p=2$; for $p=3$ it is just as immediate by
observing that $\mathfrak{g}_{\alpha} \subseteq \mathfrak{g}(h, i)$ if and only if $\mathfrak{g}_{-\alpha} \subseteq \mathfrak{g}(h, -i)$, for $\alpha \in \Phi$ and $i \in \mathbb{F}_3$. Apart from the root system of type $E_8$ in characteristic $p=5$, this settles the classification in
bad characteristic for Lie algebras of exceptional type.

The strategy is based on case-by-case computations on Richardson orbits; \ref{toralRichorbit} follows closely Premet (\cite{ModularMorozov}); the goal is to adapt some of his considerations to our setting. We will assume that \textit{$p$ is a good prime for the root system of
$\mathfrak{g}$}. After working on this case, we will separately examine type $E_8$ in characteristic $p=5$, since it requires some slight modifications to the strategy.

\subsection{The Richardson orbit associated to a balanced toral element} \label{toralRichorbit} \hfill\break

There is a canonical way of associating a Richardson orbit to a $G$-conjugacy class of toral elements. Let $h$ be a representative of one of these orbits; up to conjugacy it belongs to the Lie algebra $\mathfrak{t}$ of a fixed maximal torus $T \subseteq G$, and the centralizer $\mathfrak{l}=\mathfrak{c}_{\mathfrak{g}}(h)$ is a Levi subalgebra of $\mathfrak{g}$.

As the $0$-orbit on $\mathfrak{l}$ is rigid, the nilpotent orbit $\mathcal{O}_0= \mathrm{Ind}_{\mathfrak{l}}^{\mathfrak{g}}(\{0\})$ obtained by induction from $0 \in \mathfrak{l}$ is the unique nilpotent orbit contained in the sheet $\mathcal{S}$ whose open decomposition class is $\mathcal{D}(\mathfrak{l},0)= \mathrm{Ad} \, G (\mathfrak{z(l)}_{reg})$. 
%As $h$ is toral $\mathcal{O}_0 \subseteq \mathcal{N}_p(\mathfrak{g})=\{ x \in \mathcal{N}(\mathfrak{g}) \, | \, x^{[p]}=0 \}$, the restricted nullcone, but we will not make use of this.

If in addition $h$ is balanced, there are some restrictions on the values $\dim \mathcal{O}_0$ can assume. Let us consider $e \in \mathcal{O}_0$ which lies in the same sheet as $h$. It verifies $\dim \mathfrak{g}_e = \dim G_e = \dim G_h = \dim \mathfrak{g}_h$. The following equalities hold:
\vspace{-2mm}
\begin{multline*}
 \dim \mathcal{O}_0 = \dim \mathfrak{g} - \dim G_e = \\
= \dim \mathfrak{g} - \dim \mathfrak{g}_h = \sum_{i \in \mathbb{F}_{p}^{\times}} \dim \mathfrak{g}(h,i) = (p-1) \cdot \dim \mathfrak{g}(h,1).
\end{multline*}

\noindent Hence $\dim \mathcal{O}_0  \equiv 0 \mod (p-1)$, so that a balanced toral element $h$ naturally corresponds to a Richardson orbit $\mathcal{O}_0$ of dimension divisible by $p-1$.

\subsection{Structure of the centralizer.} \label{List}

Let $h$ be a balanced toral element, without loss of generality we can assume it corresponds to a choice of Kac coordinates $(a_0, a_1, \ldots, a_l)$ as in \ref{ToralKac}. Because of the limitations on its dimension, there there are not too many options for the orbit $\mathcal{O}_0$ associated to it as in \ref{toralRichorbit}. We make use of \cite[Tables 6-10]{InducedExceptionalOrbits}; the sheet diagram describes the structure of the centralizer $\mathfrak{l}$ of $h$. We need to consider all possible ways of embedding a centralizer of the relevant type in the Dynkin diagram, since Kac coordinates for a balanced element might be relative to a standard Levi subalgebra conjugated to the centralizer we can read off from the tables. For exceptional Lie algebras two isomorphic Levi subalgebras are always conjugated, except for three cases in type $E_7$. Still, what is chiefly needed is the dimension of the orbit $\mathrm{Ind}^\mathfrak{g}_\mathfrak{l}(0)$, that depends only on the type of the centralizer $\mathfrak{l}$. Finally, simple roots belonging to a base for the standard Levi subalgebra $\mathfrak{l}$ can be chosen up to a permutation of $\Delta$ corresponding to an element of $\Omega$.

\noindent For every simple Lie algebra of exceptional type, we produce the list:

\begin{center}
\begin{tabular}{|c|c|}
\hline
$\Omega$-orbits of standard Levi subalgebras $\overline{\mathfrak{l}}$ & $\dim \mathrm{Ind}^\mathfrak{g}_\mathfrak{l}(0)$ \\
\hline
\end{tabular}
\end{center}

\subsection{Kac coordinates for $h$.}

If $h$ is a balanced element in $\mathrm{char} k = p>0$, $\dim \mathcal{O}_0$ is divisible by $p-1$ (\ref{toralRichorbit}). All the potential centralizers of $h$ are those for which $\dim \mathrm{Ind}^\mathfrak{g}_\mathfrak{l}(0)$ is divisible by $p-1$, and so they can be read off from the list as in \ref{List}. Any of these choices of centralizers tells exactly which Kac coordinates are equal to $0$, while all the others must be strictly positive, let them be $a_0, a_{i_1}, \ldots, a_{i_r} $. Notice that $a_0>0$ because the characteristic is good for the root system. We inspect all possible strictly positive coordinates that satisfy the condition $a_0= p - \sum_{j=1}^{r} b_{i_j} a_{i_j}$. Each one of the $a_i$ is the eigenvalue of $\mathrm{ad} h$ on $\mathfrak{g}_{\alpha_i}$, so by linearity one obtains the eigenvalue of $\mathrm{ad} h$ on any root space $\mathfrak{g}_\alpha$, for $\alpha \in \Phi$. Finally, we check that $\dim \mathfrak{g}(h,i)$ for $i \in \mathbb{F}_p^\times$ is independent of $i >0$.

\subsection{Restrictions on the characteristic.} \label{restrictions}

We repeat this procedure varying the characteristic $p>0$. There is always an upper bound for a prime number $p$ such that a balanced toral element in characteristic $p$ can potentially exist, given the type of the Lie algebra. Indeed, such a value depends on the root system: \textit{e.g.} $p \leq | \Phi |+1$, but this can actually be refined. Observe that $| \Phi |$ equals the dimension of the regular nilpotent orbit for the Lie algebra $\mathfrak{g}$, and assume $\frac{| \Phi |}{2} +1 < p< | \Phi |+1$. If a balanced toral element $h$ exists in characteristic $p$, all its eigenspaces must be of dimension $1$. Moreover, the orbit $\mathcal{O}_0$ corresponding to $h$ is not the regular nilpotent orbit. Thus $h$ is not a regular semisimple element, so its centralizer contains a nonzero root space $\mathfrak{g}_{\alpha}$, and without loss of generality $\alpha \in \Delta$. Otherwise stated, in the sheet diagram of $\mathcal{O}_0$ there is at least one node with label $0$. We apply a classical result of Bourbaki (\cite[VI.1.6 Corollaire 3 \`{a} la Proposition 19]{ErBourba}) to show that for at least one $i \in \mathbb{F}_p^\times$ we have $\dim \mathfrak{g}(h,i)>1$. There exist simple roots $\alpha, \beta \in \Delta$ such that $\alpha(h)=0$, $\beta(h) >0$ and $\alpha +\beta$ is a root; this is clear considering two adjacent nodes on the sheet diagram such that one is labelled $0$ and the other has label $2$. But then $\beta, \alpha+\beta \subseteq \mathfrak{g}(h,\beta(h))$, and so this subspace has dimension strictly greater than $1$.

In any case, the subregular orbit is the only one whose sheet diagram contains exactly one node labelled $0$, and its dimension is $| \Phi |-2$. Using an argument similar to the one above, if neither $| \Phi |$ nor $| \Phi |-2$ is of the form $2(q-1)$ for a prime $q$, then we can reduce even further the number of primes to check: it is not necessary to consider the range \mbox{$\frac{| \Phi |}{3}+ 1 < p < | \Phi |+1$}.

As a result, for every root system it is enough to check only the following primes:
\begin{itemize}
\item Type $G_2$, $5 \leq p \leq 7$ and $p=13$.
\item Type $F_4$, $5 \leq p \leq 23$.
\item Type $E_6$, $5 \leq p \leq 37$ and $p=73$.
\item Type $E_7$, $5 \leq p \leq 43$ and $p=127$.
\item Type $E_8$, $7 \leq p \leq 79$ and $p=241$.
\end{itemize}

Except for type $F_4$, $| \Phi |+1$ is a prime number, hence it is necessary to check if there exist balanced toral elements in characteristic $| \Phi |+1$, that is, an element $h$ whose adjoint action on the Lie algebra yields all nonzero eigenspaces of dimension $1$. This represents the trickiest case for every root system, and it turns out that only in type $G_2$ there exist balanced elements in characteristic $| \Phi |+1$.

\subsection{The algorithm} \hfill\break

The code we wrote is absolutely elementary and does not make use of any functions in the C$++$ library. It varies slightly for each distinct type as we have to input two sets of data: the root system and the sheet diagram of Richardson orbits.

\subsubsection{The root system.} In order to obtain the root system we use some basic routines in GAP (\cite{GAP}). What we need is a system of positive roots. This is easily achieved by the following sequence of commands, relative to type $G_2$ in the example:

\vspace{2mm}
\texttt{>L:= SimpleLieAlgebra( "G", 2, Rationals );;}

\verb">R:= RootSystem( L );;"

\verb">PositiveRoots( R );"

\verb">SimpleSystem( R );"

\vspace{2mm}
The second last command returns a list of positive roots $v_1, \ldots, v_{\frac{| \Phi |}{2}}$; these are column vectors. The last command provides the subset of roots $v_1, \ldots, v_n$ (where $n= \mathrm{rank} \mathfrak{g}$) that are a basis of simple roots for the root system above.

In order to simplify the computations ahead, it is convenient to base change the vectors $v_1, \ldots, v_n$ to the standard basis $(1, 0, \ldots, 0), (0,1, \ldots, 0),$ $ \ldots, (0, \ldots, 0,1)$; these correspond to the roots $\alpha_1, \ldots, \alpha_n$ (Bourbaki's numbering) respectively. This we did by considering the $n \times n$ matrix $A=(v_1 \ldots v_n)$. For type $F_4$ we used the matrix $A=(v_2 v_4 v_3 v_1)$ to match the required order. Then we computed the inverse $A^{-1}$ and our list of positive roots is the set of vectors $\{ A^{-1}v_1, A^{-1}v_2, \ldots, A^{-1}v_{\frac{| \Phi |}{2}} \} $.

\subsubsection{The pseudocode.} \label{algorithm}

Call $n= \mathrm{rank} \, \mathfrak{g}$ and $q=$ number of Richardson orbits.

First of all, we initialize the integers $s_0=1, s_1, \ldots, s_n$ representing the coefficients of the highest roots.

The positive roots are given in C in the form a matrix $M$. This is a $\frac{| \Phi |}{2} \times n$ matrix whose $i^{\mathrm{th}}$ row is the vector $A^{-1}v_i$. For instance, when $\mathfrak{g}$ is of type $G_2$ the matrix \texttt{M} is:

\vspace{2mm}
\texttt{M[0][0]=1; M[0][1]=0;}  $\qquad \alpha_1$

\texttt{M[1][0]=0; M[1][1]=1;} $\qquad \alpha_2$

\texttt{M[2][0]=1; M[2][1]=1;} $\qquad \alpha_1+\alpha_2$

\texttt{M[3][0]=2; M[3][1]=1;} $\qquad 2\alpha_1+\alpha_2$

\texttt{M[4][0]=3; M[4][1]=1;} $\qquad 3\alpha_1+\alpha_2$

\texttt{M[5][0]=3; M[5][1]=2;} $\qquad 3\alpha_1+2\alpha_2$.

\vspace{3mm}

The list as in \ref{List} is inputted in the form of a $q \times (n+1)$ matrix $O$. Data for these orbits are sourced from the Tables in \cite{InducedExceptionalOrbits}. As for the first $n$ entries in each row, simple roots generating the centralizer Levi subalgebra correspond to $0$-entries in the matrix, whereas all the others are given the value $2$ (this could have been any nonzero value, we chose $2$ in analogy with the sheet diagram). For example, in type $F_4$ a centralizer generated by roots $\alpha_1$ and $\alpha_2$ (in Bourbaki's notation) is associated to the $i$-th row of the matrix in this way: $O_{i,1}=0 O_{i,2}=0 O_{i,3}=2 O_{i,4}=2$. 

On the other hand, the last column of the matrix $O$ contains the dimension of the orbit considered. Carrying on with the same example, we have $O_{i,5}=42$, as one can read off from \cite[Table $10$]{InducedExceptionalOrbits} (indeed, the orbit whose sheet diagram is $0022$ has dimension $42$).

\paragraph{} The characteristic $p$ varies according to our discussion in \ref{restrictions}. We initialize the range of primes in the form of an array $P=(P_1, \ldots, P_r)$. Let $p=P_i$ be fixed. The program finds all the rows of the matrix $O$ whose entry in the last column is divisible by $p-1$. The other entries in the same row describe the structure of the centralizer. Kac coordinates are contained in an array $a$. In case $O_{i,j}=0$ we set $a_j=0$, while the remaining entries of $a$ (including $a_0$) must be strictly positive.

The code examines all the tuples with $a=p-\sum_{i=1}^{n}s_i a_i$  over all the suitable $a_i$. For each of them an array $v$ of dimension $| \Phi |$ has the task of counting the dimension of the different eigenspaces $\mathfrak{g}(h,i)$. This is easy to compute: if the root $\alpha$ corresponds to the $j$-th row of $M$, then the root space $\mathfrak{g}_\alpha$ has eigenvalue $j=\sum_{k=1}^{\mathrm{rank} \, \mathfrak{g}}M_{j,k}a_k$ , and the component $v_j$ is increased.

Finally, the program returns tuples of Kac coordinates for which $v_i=v_1$ for all $i \in \mathbb{F}^\times_p$.

\vspace{2mm}

\noindent Here is the pseudocode for the algorithm described.

\vspace{2mm}

\begin{algorithm}[H]
\caption{Kac coordinates of balanced toral elements}
\DontPrintSemicolon
\BlankLine
Initialize with $n, q,r, s_0, \ldots, s_n, M, P, O$\;
Define $a=0$, $v=0$\;
Define $p,i,j,m$ dummy indices\;
\For{$p \leftarrow P_1$ to $P_r$}{
 \For{$m \leftarrow 1$ to $q$}{
  \If{$O_{n+1,m} \equiv 0 \; \mathrm{mod} \; p-1$}{
  Set $v=0$ array\;
  Set $a=0$ array\;
  
  \For{$i \leftarrow 1$ to $n$}{
  \If{$O_{i,m} \neq 0$}{
  Set $a_i=1$\;
  }
  }
  \While{$s_1 a_1+ \ldots s_n a_n <p$}{
  Set $a_0=p-(s_1 a_1+ \ldots s_n a_n)$\;
  \For{$i \leftarrow 1$ to $| \Phi |$}{
  Set $j=M_{i,1}a_1+ \ldots+ M_{i,n}a_n$\;
  Increase $v_j$ by $1$\;
  Increase $v_{p-j}$ by $1$\;
  }
  \If{$v_1=v_2= \ldots =v_{p-1}$}{
  Return characteristic $p$\;
  Return Kac coordinates $a_0, a_1, \ldots, a_n$\;
  }
  }
  }
 
 } 
} 
\end{algorithm}

\subsection{The bad characteristic case} \hfill \break

Let us now focus on type $E_8$ in characteristic $5$. Unlike the good characteristic case, here the centralizer of a semisimple element need not be a Levi subalgebra, more generally it is a \textit{pseudo-Levi subalgebra}. Kac coordinates can still be used to parameterize $G$-classes of toral elements, yet we have to allow $a_0=0$. As the extended Dynkin diagram in type $E_8$ does not admit any symmetry, our list as in \ref{List} simply contains all the potential centralizers of a toral element. For $p=5$, we check only the possible centralizers, not the dimension of the orbit (nilpotent orbits may be different from good characteristic, and the description of sheets in \ref{rigiddec} does not apply). However, we stress that this does not hinder the efficiency of the algorithm, mainly because there are very few centralizers that can work for a toral element as in \ref{ToralKac}, as Kac coordinates are nonnegative integers satisfying $a_0=5-2a_1-3a_2-4a_3-6a_4-5a_5-4a_6-3a_7-2a_8$.

The Tables in the next section contain all conjugacy classes of toral elements in characteristic $2$ and $3$ as well; elements in these classes are automatically balanced. We simply list Kac coordinates satisfying (\ref{Kaccoo}) (for $m=2$ and $3$) up to the action of $\Omega$, since they are not too many.

\section{The Tables} \label{Tables}

We implemented the algorithm discussed in \ref{algorithm} using the programming language C. The Tables hereafter represent all the $G$-conjugacy classes of balanced toral elements and are sorted by type of the root system and characteristic (noted $p$). Every $G$-conjugacy class is described by a choice of Kac coordinates. Let $h$ be an element in a fixed class, and note $\Phi_0$ the root system of the centralizer of $h$. The second last column contains the type of $\Phi_0$; notice that in a few cases $h$ is regular semisimple. Finally, the value in last column is the dimension $d=\dim \mathfrak{g}(h,i)$ of each $h$-eigenspace for $i \neq 0$.

\section*{Type $G_2$}

\begin{longtable}{ |c|c|c|c|c|c| }
%\begin{center}
%\begin{tabular}{ |c|c|c|c|c|c| }
 \hline
 \multirow{2}{1em}{  $p$} & \multicolumn{3}{|c|}{Kac coord.} & \multirow{2}{3em}{ Type of $\Phi_0$} & \multirow{2}{1em}{  $d$}\\
  & $a_0$ & $a_1$ & $a_2$ & &\\
  \hline
  $2$ & $0$ & $1$ & $0$ & $\widetilde{A}_1A_1$ & $8$ \\
 $3$ & $0$ & $0$ & $1$ & $A_2$ & $3$  \\
 $3$ & $1$ & $1$ & $0$ & $\widetilde{A}_1$ & $5$  \\
 $7$ & $2$ & $1$ & $1$ & regular & $2$ \\
 $13$ & $6$ & $1$ & $2$ & regular & $1$  \\
 $13$ & $1$ & $2$ & $3$ & regular & $1$ \\
\hline
%\end{tabular}
\caption{}%*{Table $1$.}
\label{G2}
%\end{center}
\end{longtable}

\section*{Type $F_4$}

\begin{longtable}{ |c|c|c|c|c|c|c|c| }
%\begin{center}
%\begin{tabular}{ |c|c|c|c|c|c|c|c| }
 \hline
 \multirow{2}{1em}{  $p$} & \multicolumn{5}{|c|}{Kac coordinates} & \multirow{2}{3em}{ Type of $\Phi_0$} & \multirow{2}{1em}{  $d$}\\
  & $a_0$ & $a_1$ & $a_2$ & $a_3$ & $a_4$ && \\
  \hline
  $2$ & $0$ & $1$ & $0$ & $0$ & $0$ & $A_1C_3$ & $28$ \\
  $2$ & $0$ & $0$ & $0$ & $0$ & $1$ & $B_4$ & $16$ \\
  $3$ & $1$ & $1$ & $0$ & $0$ & $0$ & $C_3$ & $15$ \\
  $3$ & $0$ & $0$ & $1$ & $0$ & $0$ & $\widetilde{A}_2A_2$ & $18$ \\
  $3$ & $1$ & $0$ & $0$ & $0$ & $1$ & $B_3$ & $15$ \\
 $5$ & $1$ & $1$ & $0$ & $0$ & $1$ & $B_2$ & $10$ \\
 $7$ & $1$ & $0$ & $0$ & $1$ & $1$ & $A_2$ & $7$ \\
 $7$ & $2$ & $1$ & $1$ & $0$ & $0$ & $\widetilde{A}_2$ & $7$ \\
 $13$ & $2$ & $1$ & $1$ & $1$ & $1$ & regular & $4$ \\
 $17$ & $4$ & $2$ & $1$ & $1$ & $1$ & regular & $3$ \\
 $17$ & $1$ & $1$ & $2$ & $1$ & $2$ & regular & $3$ \\
\hline
%\end{tabular}
\caption{}%*{Table $2$.}
\label{F4}
%\end{center}
\end{longtable}

\section*{Type $E_6$}

\begin{longtable}{ |c|c|c|c|c|c|c|c|c|c| }
%\begin{center}
%\begin{tabular}{ |c|c|c|c|c|c|c|c|c|c| }
 \hline
 \multirow{2}{1em}{  $p$} & \multicolumn{7}{|c|}{Kac coordinates} & \multirow{2}{3em}{ Type of $\Phi_0$} & \multirow{2}{1em}{  $d$} \\
  & $a_0$ & $a_1$ & $a_2$ & $a_3$ & $a_4$ & $a_5$ & $a_6$ && \\
  \hline
 $2$ & $1$ & $1$ & $0$ & $0$ & $0$ & $0$ & $0$ & $D_5$ & $32$ \\
 $2$ & $0$ & $0$ & $0$ & $1$ & $0$ & $0$ & $0$ & $A_1A_5$ & $40$ \\
 $3$ & $2$ & $1$ & $0$ & $0$ & $0$ & $0$ & $0$ & $D_5$ & $16$ \\
 $3$ & $1$ & $2$ & $0$ & $0$ & $0$ & $0$ & $0$ & $D_5$ & $16$ \\
 $3$ & $1$ & $1$ & $0$ & $0$ & $0$ & $0$ & $1$ & $D_4$ & $24$ \\
 $3$ & $1$ & $0$ & $1$ & $0$ & $0$ & $0$ & $0$ & $A_5$ & $21$ \\
 $3$ & $1$ & $0$ & $0$ & $1$ & $0$ & $0$ & $0$ & $A_1A_4$ & $25$ \\
 $3$ & $1$ & $0$ & $0$ & $0$ & $0$ & $1$ & $0$ & $A_1A_4$ & $25$ \\
 $3$ & $0$ & $0$ & $0$ & $0$ & $1$ & $0$ & $0$ & $A_2^3$ & $27$ \\
 $5$ & $1$ & $1$ & $1$ & $0$ & $0$ & $0$ & $1$ & $A_3$ & $15$ \\
$7$ & $2$ & $4$ & $0$ & $0$ & $0$ & $0$ & $1$ & $D_4$ & $8$ \\
$7$ & $1$ & $4$ & $0$ & $0$ & $0$ & $0$ & $2$ & $D_4$ & $8$ \\
$7$ & $2$ & $0$ & $1$ & $0$ & $1$ & $0$ & $0$ & $A_2^2$ & $10$ \\
$7$ & $1$ & $1$ & $1$ & $0$ & $0$ & $1$ & $1$ & $A_2$ & $11$ \\
$11$ & $2$ & $1$ & $1$ & $0$ & $1$ & $1$ & $1$ & $A_1$ & $7$ \\
$11$ & $1$ & $1$ & $1$ & $0$ & $1$ & $1$ & $2$ & $A_1$ & $7$ \\
$13$ & $2$ & $1$ & $1$ & $1$ & $1$ & $1$ & $1$ & regular & $6$ \\
$19$ & $4$ & $1$ & $2$ & $2$ & $1$ & $1$ & $1$ & regular & $4$ \\
$19$ & $1$ & $4$ & $2$ & $2$ & $1$ & $1$ & $1$ & regular & $4$ \\
$37$ & $9$ & $5$ & $4$ & $3$ & $2$ & $1$ & $1$ & regular & $2$ \\
$37$ & $5$ & $9$ & $3$ & $4$ & $2$ & $1$ & $1$ & regular & $2$ \\
$37$ & $8$ & $9$ & $4$ & $1$ & $1$ & $2$ & $3$ & regular & $2$ \\
$37$ & $9$ & $8$ & $1$ & $4$ & $1$ & $2$ & $3$ & regular & $2$ \\
\hline
%\end{tabular}
\caption{}%*{Table $3$.}
\label{E6}
%\end{center}
\end{longtable}

\section*{Type $E_7$}

\begin{longtable}{ |c|c|c|c|c|c|c|c|c|c|c| }
%\begin{center}
%\begin{tabular}{ |c|c|c|c|c|c|c|c|c|c|c| }
 \hline
 \multirow{2}{1em}{  $p$} & \multicolumn{8}{|c|}{Kac coordinates} & \multirow{2}{3em}{ Type of $\Phi_0$} & \multirow{2}{1em}{  $d$} \\
  & $a_0$ & $a_1$ & $a_2$ & $a_3$ & $a_4$ & $a_5$ & $a_6$ & $a_7$ && \\
  \hline
 $2$ & $1$ & $0$ & $0$ & $0$ & $0$ & $0$ & $0$ & $1$ & $E_6$ & $54$ \\
 $2$ & $0$ & $1$ & $0$ & $0$ & $0$ & $0$ & $0$ & $0$ & $A_1D_6$ & $64$ \\
 $2$ & $0$ & $0$ & $1$ & $0$ & $0$ & $0$ & $0$ & $0$ & $A_7$ & $70$ \\
 $3$ & $2$ & $0$ & $0$ & $0$ & $0$ & $0$ & $0$ & $1$ & $E_6$ & $27$ \\
 $3$ & $1$ & $0$ & $0$ & $0$ & $0$ & $0$ & $1$ & $0$ & $A_1D_5$ & $42$ \\
 $3$ & $1$ & $0$ & $1$ & $0$ & $0$ & $0$ & $0$ & $0$ & $A_6$ & $42$ \\
 $3$ & $1$ & $1$ & $0$ & $0$ & $0$ & $0$ & $0$ & $0$ & $D_6$ & $33$ \\
 $3$ & $0$ & $0$ & $0$ & $1$ & $0$ & $0$ & $0$ & $0$ & $A_2A_5$ & $45$ \\
$5$ & $1$ & $1$ & $0$ & $0$ & $0$ & $0$ & $1$ & $0$ & $D_4 A_1$ & $25$ \\
$7$ & $2$ & $1$ & $0$ & $1$ & $0$ & $0$ & $0$ & $0$ & $A_5$ & $16$ \\
$7$ & $2$ & $1$ & $0$ & $0$ & $0$ & $0$ & $1$ & $1$ & $D_4$ & $17$ \\
$7$ & $1$ & $0$ & $0$ & $0$ & $1$ & $0$ & $1$ & $0$ & $A_1^3 A_2$ & $19$ \\
$11$ & $1$ & $1$ & $1$ & $0$ & $0$ & $1$ & $1$ & $1$ & $A_2$ & $12$ \\
$13$ & $2$ & $1$ & $0$ & $1$ & $1$ & $0$ & $1$ & $0$ & $A_1^3$ & $10$ \\
$19$ & $2$ & $1$ & $1$ & $1$ & $1$ & $1$ & $1$ & $1$ & regular & $7$ \\
\hline
%\end{tabular}
\caption{}%*{Table $4$.}
\label{E7}
%\end{center}
\end{longtable}

\section*{Type $E_8$}

\begin{longtable}{ |c|c|c|c|c|c|c|c|c|c|c|c| }
%\begin{center}
%\begin{tabular}{ |c|c|c|c|c|c|c|c|c|c|c|c| }
 \hline
 \multirow{2}{1em}{  $p$} & \multicolumn{9}{|c|}{Kac coordinates} & \multirow{2}{3em}{ Type of $\Phi_0$} & \multirow{2}{1em}{  $d$}\\
  & $a_0$ & $a_1$ & $a_2$ & $a_3$ & $a_4$ & $a_5$ & $a_6$ & $a_7$ & $a_8$ && \\
  \hline
  $2$ & $0$ & $1$ & $0$ & $0$ & $0$ & $0$ & $0$ & $0$ & $0$ & $D_8$ & $128$ \\
  $2$ & $0$ & $0$ & $0$ & $0$ & $0$ & $0$ & $0$ & $0$ & $1$ & $A_1E_7$ & $112$ \\
  $3$ & $1$ & $1$ & $0$ & $0$ & $0$ & $0$ & $0$ & $0$ & $0$ & $D_7$ & $78$ \\
  $3$ & $0$ & $0$ & $1$ & $0$ & $0$ & $0$ & $0$ & $0$ & $0$ & $A_8$ & $84$ \\
  $3$ & $0$ & $0$ & $0$ & $0$ & $0$ & $0$ & $0$ & $1$ & $0$ & $A_2E_6$ & $81$ \\
  $3$ & $1$ & $0$ & $0$ & $0$ & $0$ & $0$ & $0$ & $0$ & $1$ & $E_7$ & $57$ \\
 $5$ & $1$ & $1$ & $0$ & $0$ & $0$ & $0$ & $0$ & $0$ & $1$ & $D_6$ & $45$ \\
 $5$ & $1$ & $0$ & $0$ & $1$ & $0$ & $0$ & $0$ & $0$ & $0$ & $A_1A_6$ & $49$ \\
 $5$ & $0$ & $0$ & $1$ & $0$ & $0$ & $0$ & $0$ & $0$ & $1$ & $A_1A_6$ & $49$ \\
 $5$ & $0$ & $0$ & $0$ & $0$ & $0$ & $1$ & $0$ & $0$ & $0$ & $A_4^2$ & $50$ \\
$7$ & $2$ & $0$ & $0$ & $0$ & $0$ & $0$ & $0$ & $1$ & $1$ & $E_6$ & $28$ \\
$7$ & $1$ & $1$ & $0$ & $0$ & $0$ & $0$ & $1$ & $0$ & $0$ & $D_4A_2$ & $35$ \\
$11$ & $1$ & $1$ & $1$ & $0$ & $0$ & $0$ & $0$ & $1$ & $1$ & $A_4$ & $22$ \\
$13$ & $2$ & $1$ & $0$ & $0$ & $0$ & $0$ & $1$ & $1$ & $1$ & $D_4$ & $18$ \\
$13$ & $1$ & $1$ & $1$ & $0$ & $0$ & $1$ & $0$ & $0$ & $1$ & $A_2^2$ & $19$ \\
$19$ & $1$ & $1$ & $1$ & $1$ & $0$ & $0$ & $1$ & $1$ & $1$ & $A_2$ & $13$ \\
$31$ & $2$ & $1$ & $1$ & $1$ & $1$ & $1$ & $1$ & $1$ & $1$ & regular & $8$ \\
$41$ & $5$ & $3$ & $2$ & $1$ & $1$ & $1$ & $1$ & $1$ & $1$ & regular & $6$ \\
$41$ & $1$ & $2$ & $1$ & $1$ & $1$ & $1$ & $1$ & $2$ & $4$ & regular & $6$ \\
$61$ & $7$ & $5$ & $1$ & $3$ & $1$ & $1$ & $1$ & $2$ & $4$ & regular & $4$ \\
$61$ & $3$ & $2$ & $2$ & $1$ & $1$ & $1$ & $4$ & $1$ & $7$ & regular & $4$ \\
\hline
%\end{tabular}
\caption{}%*{Table $5$.}
\label{E8}
%\end{center}
\end{longtable}

\section{Conjugacy of balanced toral elements up to scalar multiples} \label{ScalarMultiples}

Let us stick to the assumption that $\mathrm{char} \, k$ is a good prime for the root system of $\mathfrak{g}$. The invariant subalgebra $k[\mathfrak{g}]^G$ is a polynomial ring generated by $l=\mathrm{rank} \, G$ algebraically independent homogeneous polynomials $\{ f_1, \ldots, f_l \}$, and the zero locus of the ideal generated by the $f_i$'s is the nullcone $\mathcal{N}(\mathfrak{g})$ (see, for example, \cite{JantzenNilp}).

Even though generators are not uniquely determined, the set $\{ d_1, \ldots, d_l \}$ of their degrees is (in characteristic $0$ this is a result of Chevalley, for good characteristic $p>0$ see \cite{InvariantsSymetriques}). Remarkably, the degrees in good characteristic are the same as in characteristic $0$ (\cite{InvariantsSymetriques}); they can be found, for example, in \cite{Humph-Reflections}. There is always an invariant of degree $2$, due to the existence of a $G$-invariant symmetric bilinear form $\kappa$ on $\mathfrak{g}$.

\subsection{Scalar multiples of balanced toral elements.} \label{InvMultip} \hfill \break

Throughout this section we will write $h \sim h^\prime$ to indicate that two elements $h$ and $h^\prime$ are in the same $G$-orbit. Let $h$ be the usual representative of a class of balanced toral elements as in section \ref{Tables}; for any $r \in \mathbb{F}_p^\times$ the element $rh$ is still balanced. A natural question to ask is whether $rh$ is in the same conjugacy class of $h$. Once fixed the characteristic $p>0$, if $h$ represents the unique conjugacy class with a certain type of centralizer, of course $rh \sim h$ for any $r \in \mathbb{F}_p^\times$. For any fixed $p$ there are at most two conjugacy classes with isomorphic centralizer, except for type $E_6$ in characteristic $37$, where there are four classes. The cases with exactly two conjugacy classes with the same type of centralizer are:
\begin{itemize}[align=left]
\item[$(i)$] type $G_2$, characteristic $13$;
\item[$(ii)$] type $F_4$, characteristic $17$;
\item[$(iii)$] type $E_6$, characteristic $3$, centralizer of type $D_5$;
\item[$(iv)$] type $E_6$, characteristic $3$, centralizer of type $A_1A_4$;
\item[$(v)$] type $E_6$, characteristic $7$;
\item[$(vi)$] type $E_6$, characteristic $11$;
\item[$(vii)$] type $E_6$, characteristic $19$;
\item[$(viii)$] type $E_8$, characteristic $5$;
\item[$(ix)$] type $E_8$, characteristic $41$;
\item[$(x)$] type $E_8$, characteristic $61$.
\end{itemize}

Since $h$ is semisimple, at least one of the invariant polynomials does not vanish on $h$; assume for instance that $f_i$ has such property. Notice that the invariant of degree $2$ always annihilates a balanced toral element $h$. This is because, up to replacing the invariant form $\kappa$ with a scalar multiple, one can write:
\begin{displaymath}
\kappa(h,h) = \mathrm{Tr} (\mathrm{ad} h)^2 = \sum_{i \in \mathbb{F}_p^\times} \dim \mathfrak{g}(h,i) i^2 = \dim \mathfrak{g}(h,1) \sum_{i \in \mathbb{F}_p^\times} i^2 \equiv 0 \; \mathrm{mod} \, p.
\end{displaymath}

\subsubsection{}
The elements of $\mathbb{F}_p^\times$ are all the $r \in k$ satisfying $r^{p-1}=1$. Let us assume that $rh \sim h$ for all $r \in \mathbb{F}_p^\times$, and suppose $f_i(h) \neq 0$. By invariance $f_i(h)=f_i(rh)=r^{d_i} f_i(h)$, and so $r^{d_i} =1$ for all $r \in \mathbb{F}_p^\times$. The number of solutions of the polynomial $x^{d_i} =1$ over the field $\mathbb{F}_p$ is at most $d_i$. Therefore, if $p-1 > \max_{i=1, \ldots, l} d_i$, there exists a multiple $rh$ of $h$ which is not conjugate to $h$.
When there are two conjugacy classes with representatives $h$ and $h^\prime$, this implies that at least one multiple of $h$ is conjugate to $h^\prime$ and viceversa.

Since maxima for the degrees of the $f_i$'s in type $G_2, F_4, E_6$ and $E_8$ are respectively $6, 12, 12$ and $30$, this argument suffices to settle the problem in all cases with two conjugacy classes except for $(iii), (iv),(v),(vi)$ and $(viii)$.

\paragraph{}
A closer look shows that case $(vi)$ can actually be tackled in the same way. Indeed, consider the root system of type $E_6$ in characteristic $11$. There are two classes of balanced toral elements with representatives $h$ and $h^\prime$, whose centralizer is in both cases of type $A_1$. Assume $rh$ is conjugate to $h$ for all $r \in \mathbb{F}_{11}^\times$. The degrees of invariants in type $E_6$ are $2, 5, 6, 8, 9, 12$. The argument above shows that $f_1(h)= f_2(h)=f_3(h)=f_4(h)=f_5(h)=0$. But $f_6(h) \neq 0$ implies $r^{12}=1$ for all $r \in \mathbb{F}_{11}^\times$, and so $r=\pm 1$.

\subsubsection{Type $E_6$ in characteristic $37$.}
Only in this case are there four conjugacy classes with isomorphic centralizers, let their representatives be $h_1, h_2, h_3, h_4$. Since the maximum degree for a generator of the invariant subalgebra is $12 < 36$, none of the $h_i$'s has the property that all its multiples are conjugate to $h_i$ itself.

Consider the set $X_i = \{ r \in \mathbb{F}_{37}^\times | rh_1 \sim h_i \}$. Clearly $X_1 \neq \emptyset$ and $\mathbb{F}_{37}^\times= \coprod_{i=1}^4 X_i$. Assume $X_j \neq \emptyset$ for $j \neq 1$. For $s_j \in X_j$ we have the inclusion $s_j X_1 \subseteq X_j$. This means that there exists $g \in G$ such that $s_j h_1 = (\mathrm{Ad} g) h_j$, and so $(\mathrm{Ad} g) s_j^{-1} h_j = h_1$. In particular, for any $t \in X_j$, $s_j^{-1} t h_1 \sim s_j^{-1} h_j \sim h_1$. Otherwise stated, $s_j^{-1} X_j \subseteq X_1$.

The two inclusions imply that $|X_1| =|X_j|$ for every $X_j \neq \emptyset$. Then $\mathbb{F}_{37}^\times$ is the disjoint union of at least two sets of the same cardinality.
Notice that we cannot have $|X_1|=12$. In that case we could assume without loss of generality $h_1 \sim h_2 \sim h_3$. But that would lead to $r h_4 \sim h_4$ for all $r \in \mathbb{F}_{37}^\times$.

Finally, it is possible to exclude the case in which exactly two of the $X_i$'s are nonempty. Each of them should have cardinality $18$, but the highest degree of an invariant generator in type $E_6$ is $12$, and so $|X_1| \leq 12$. Then for every choice of $h_i$ and $h_j$, there exists $r \in \mathbb{F}_{37}^\times$ such that $r h_i \sim h_j$.

\subsection{Remaining cases} \hfill\break

The arguments in \ref{InvMultip} allow to conclude in almost all instances that whenever there is more than one conjugacy class of balanced toral elements with isomorphic centralizer, two representatives of distinct classes are conjugated up to a scalar multiple. Something similar happens for $(iii), (v)$ and $(viii)$ with elements in distinct classes being conjugated up to multiplication by a suitable $r \in \mathbb{F}_p^\times$. However, for type $E_6$ in characteristic $3$ with centralizer of type $A_1A_4$ (case $(iv)$) it turns out that two representatives are conjugate under an element of the full automorphisms group $Aut(\mathfrak{g})$, not of its identity component (see \ref{LastRem}). 

We are going to prove these facts by finding explicit Kac coordinates for suitable scalar multiples of each element.

\subsubsection{Type $E_6$ in characteristic $7$.}

Retain notations from Section \ref{InnerTorsion}. Recall that on the vector space $V$ the symmetry with respect to the hyperplane $L_{\alpha,n} = \{ x \in V |\langle \alpha, x \rangle=n \}$ is given by $S_{\alpha,n} (x) = x + (n-\langle \alpha, x \rangle) \alpha^\vee$, where $\alpha^\vee$ is the coroot corresponding to $\alpha$.

The representative $h$ of one of the two classes (let $h^\prime$ be a representative of the other) has Kac coordinates $(2,4,0,0,0,0,1)$, and therefore it corresponds to the point $\frac{4}{7} \check{\omega}_1+\frac{1}{7} \check{\omega}_6$.  The balanced element $3h$ corresponds to $\frac{12}{7} \check{\omega}_1+\frac{3}{7} \check{\omega}_6$; we will conjugate it via the extended affine Weyl group to a point in the fundamental alcove. Translating by $- \check{\omega}_1$ we obtain the point $\frac{5}{7} \check{\omega}_1+\frac{3}{7} \check{\omega}_6$.
Let $\alpha= \alpha_1+\alpha_3+\alpha_4+\alpha_5+\alpha_6$. We apply the reflection about the hyperplane $L_{\alpha,1}$, observing that the expression of the coroot $\alpha^\vee$ in terms of the fundamental coweights is $\alpha^\vee = \check{\omega}_1-\check{\omega}_2+\check{\omega}_6$.
Then $S_{\alpha,1}(\frac{5}{7} \check{\omega}_1+\frac{3}{7} \check{\omega}_6)= \frac{5}{7} \check{\omega}_1+\frac{3}{7} \check{\omega}_6 -\frac{1}{7}(\check{\omega}_1-\check{\omega}_2+\check{\omega}_6)=\frac{4}{7} \check{\omega}_1+\frac{1}{7} \check{\omega}_2+\frac{2}{7} \check{\omega}_6$.

As for the highest root, $\widetilde{\alpha}_0^\vee = \check{\omega}_2$, and so $S_{\widetilde{\alpha}_0,1} (\frac{4}{7} \check{\omega}_1+\frac{1}{7} \check{\omega}_2+\frac{2}{7} \check{\omega}_6)= \frac{4}{7} \check{\omega}_1+\frac{2}{7} \check{\omega}_6$, which is a point in the fundamental alcove corresponding to the $7$-tuple with Kac coordinates $(1,4,0,0,0,0,2)$, so that $3h \sim h^\prime$.

\subsubsection{Type $E_8$ in characteristic $5$.}

The bad characteristic case can be treated just as the previous one. There are two conjugacy classes, with representatives $h$ and $h^\prime$, with isomorphic centralizers of type $A_1A_6$. If $h$ has Kac coordinates $(1,0,0,1,0,0,0,0,0)$, it corresponds to the point $\frac{1}{5} \check{\omega}_3$, and the element $2h$ to the point $\frac{2}{5} \check{\omega}_3$. We apply to this element, in the order, reflections about the following hyperplanes: $L_{\alpha_1+2\alpha_2+3\alpha_3+4\alpha_4+3\alpha_5+2\alpha_6+\alpha_7,1}$, $L_{ 2\alpha_1+2\alpha_2+3\alpha_3+4\alpha_4+3\alpha_5+2\alpha_6+\alpha_7+\alpha_8,1}$, $L_{2\alpha_1+2\alpha_2+4\alpha_3+5\alpha_4+4\alpha_5+3\alpha_6+2\alpha_7+\alpha_8,1}$ and finally $L_{\widetilde{\alpha}_0 - \alpha_8,1}$.

The point in the fundamental alcove thus obtained is $\frac{1}{5} \check{\omega}_2+\frac{1}{5} \check{\omega}_8$, so that $2h \sim h^\prime$.

\subsubsection{Type $E_6$ in characteristic $3$.} As for the two conjugacy classes with centralizer of type $D_5$, if $h$ is the element with Kac coordinates $(2,1,0,0,0,0,0)$, then the other is $2h$. The last case is centralizer of type $A_1A_4$; the points in the fundamental alcove corresponding to the two elements are $\frac{1}{3} \check{\omega}_3$ and $\frac{1}{3} \check{\omega}_5$. The reflection about the hyperplane $L_{\alpha_1+\alpha_2+2\alpha_3+2\alpha_4+\alpha_5,1}$ sends $\frac{2}{3} \check{\omega}_3$ to $\frac{1}{3} \check{\omega}_3$, so that $2h$ is conjugate to $h$ for elements in both these classes. This is the only case in which $G$-conjugacy does not hold for scalar multiples. Still, representatives of distinct classes are conjugated under the full automorphism group $Aut (\mathfrak{g})$ (see \ref{LastRem}).

\section{Alternative computational methods} \label{FAlternative}

For the majority of primes $p$, a manual classification of balanced toral elements is doable. Indeed, we were able to find Kac coordinates of balanced elements or disprove their existence in all cases apart from the following:

\begin{itemize}
\item root system of type $E_6$, characteristic $73$, regular nilpotent orbit;
\item root system of type $E_7$, characteristic $43$ and $127$, regular nilpotent orbit;
\item root system of type $E_8$, characteristic $79$, orbit with label $E_8(a_3)$ of dimension $234$;
\item root system of type $E_8$, characteristic $41, 61$ and $241$, regular nilpotent orbit.
\end{itemize}

We chose to resort to computational methods in order to examine these cases as well; most of them are indeed characterized by the existence of balanced toral elements. Of course, the results obtained with the code coincide with our computations in all the other cases.

In this section we want to give an idea of how one can manually obtain the classification in good positive characteristic. For some orbits this approach requires lengthy computations. Nonetheless, while a computer checks all the possible $(l+1)$-tuples of Kac coordinates as long as they satisfy $a_0= p - \sum_{i=0}^l b_i a_i$ and are compatible with the restriction on the centralizer given by the relevant nilpotent orbit, a manual approach allows to exclude straightaway most orbits or options for Kac coordinates.

\subsection{}
In good positive characteristic there exists a $G$-equivariant isomorphism of varieties between the nilpotent cone $\mathcal{N}(\mathfrak{g})$ and the unipotent variety $\mathcal{U}$ of $G$, provided the root system is not of type $A_n$ (\cite{McNinch}). As in \cite{ModularMorozov}, the Richardson orbit associated to a toral element (\ref{toralRichorbit}) is included in the restricted nullcone $\mathcal{N}_p(\mathfrak{g})$. One can then rely on the tables contained in \cite{LawtherUnipotent} to figure out the order of elements of a Richardson orbit, immediately ruling out those whose elements do not satisfy $x^{[p]}=0$.

\subsection{}
Once a Richardson orbit is fixed, it may happen that there is only one option for both the centralizer of a toral element $h$ and Kac coordinates verifying $a_0= p - \sum_{i=0}^l b_i a_i$ (up to the action of $\Omega$). Whenever this happens, there is no need to directly compute the dimension of eigenspaces $\mathfrak{g}(h,i)$ since the element $h$ is automatically balanced (\cite{ModularMorozov}).

As an example, consider the root system of type $E_8$ in characteristic $p=19$. The Richardson orbit with label $E_8(a_3)$ has dimension $234$, which is divisible by $18$; it can be induced in two ways from the $0$-orbit of a Levi subalgebra, and one of them is for centralizer of type $A_2$. There exists only one way of choosing this centralizer such that Kac coordinates satisfy $a_0= p - \sum_{i=0}^l b_i a_i$, namely when $a_4=a_5=0$. Only the $9$-tuple $(1,1,1,1,0,0,1,1,1)$ is compatible with such condition, and therefore the corresponding toral element must be balanced.

\subsection{}
Fix the Kac coordinates of a toral element $h$; whenever it is not immediately clear if $h$ is balanced or not, we directly compute the dimension of the eigenspaces. In general basic arguments of graph theory are enough to exclude some cases.

As remarked before, \cite[VI.1.6 Corollaire 3 \`{a} la Proposition 19]{ErBourba} is an exceedingly useful tool in dealing with computations on the dimension of eigenspaces. For the same purpose, we prove the following:

\begin{prop} \label{Propo}
Let $\mathfrak{g}$ be a simple Lie algebra of exceptional type with root system $\Phi$, and let $\Psi$ be the set of connected subgraphs of the extended Dynkin diagram of $\mathfrak{g}$. There exists a natural injection:
\begin{displaymath}
 \phi: \Psi \longrightarrow \Phi^+.
\end{displaymath}
\end{prop}

\begin{proof} We call $i$ the node corresponding to the simple root $\alpha_i$ for $i=1, \ldots, l$, while $0$ indicates the node relative to $-\widetilde{\alpha}_0$. Let $J \subseteq \{ 0, 1, \ldots, l \}$ be a choice of nodes for a connected subgraph of the Dynkin diagram, the root $\phi(J)$ is assigned according to the following:
\begin{displaymath}
\left\{
\begin{array}{ll}
\phi(J)= \sum_{i \in J} \alpha_i & \mathrm{if} \; 0 \notin J \\
\phi(J)= \widetilde{\alpha}_0-\sum_{i \in J \setminus \{0\}} \alpha_i & \mathrm{if} \; 0 \in J.
\end{array}
\right.
\end{displaymath}

\noindent In both cases $\phi(J) \in \Phi^+$; we only need to prove that for $\mathfrak{g}$ of exceptional type $\phi$ is injective. Clearly $\phi$ is injective when restricted to each of the subsets $\Psi^\prime=\{ J \in \Psi | 0 \notin J \}$ and $\Psi \setminus \Psi^\prime$. For Lie algebras of type $F_4, G_2, E_7$ and $E_8$, injectivity of $\Psi$ is due to:
$$
\mathrm{ht} (\phi(J^\prime)) \leq l < h-l-1 \leq \mathrm{ht} (\phi(J)),
$$
whenever $J^\prime \in \Psi^\prime, J \in \Psi \setminus \Psi^\prime$, and where $h$ is the Coxeter number.

\noindent For type $E_6$ it is enough to note that, retaining assumptions on $J$ and $J^\prime$, if $\phi(J)= \sum_{i =1}^l n_i \alpha_i$ and $\phi(J^\prime)= \sum_{i =1}^l n^\prime_i \alpha_i$, then $n^\prime_4 < 2 \leq n_4$.
\end{proof}

\begin{cor} \label{Coro}
In type $E_6$, there exists a bijective correspondence between $\Phi^+$ and the set $\Psi$ of connected subgraphs of the extended Dynkin diagram.
\end{cor}

\begin{proof}
In type $E_6$ the sets $\Phi^+$ and $\Psi$ have the same cardinality.
\end{proof}

It is possible to use Proposition \ref{Propo} and Corollary \ref{Coro} when trying to compute the dimension of eigenspaces $\mathfrak{g}(h,i)$ for a toral element $h$. Consider the extended Dynking diagram of a Lie algebra of exceptional type, labelled with the Kac coordinates of $h$; we will always use Bourbaki's numbering when referring to simple roots. The proof of \ref{Propo} implies that:
\begin{itemize}
 \item $\mathfrak{g}_{\phi(J)} \subseteq \mathfrak{g}(h, \sum_{i \in J} a_i) $ for $J \in \Psi^\prime$;
  \item $\mathfrak{g}_{-\phi(J)} \subseteq \mathfrak{g}(h, \sum_{i \in J} a_i) $ for $J \in \Psi \setminus \Psi^\prime$.
\end{itemize}

\noindent For $A, B \subseteq \{0, 1, \ldots, l \}$ with $A \cap B = \emptyset$ we define the set:
\begin{displaymath}
{}_A \Psi_B = \left\{ J \in \Psi \; | \; A \subseteq J \; \mathrm{and} \; J \cap B = \emptyset \right\}.
\end{displaymath}

\noindent When it cannot be source of confusion, we will drop the set notation, for example we will write ${}_1 \Psi_4$ instead of ${}_{ \{1\} }\Psi_{ \{4\} }$.

Let us assume $C \subseteq \{ 1, \ldots, l \}$ is the subset of nodes of the Dynkin diagram  corresponding to the simple roots in the centraliser of $h$. If $A \cup B = \{ 0, \ldots, l \} - C$ and $A \cap B = \emptyset$, consider $J \in {}_A \Psi_B$. By Proposition \ref{Propo}, if $0 \in B$ we have $\mathfrak{g}_{\phi(J)} \subseteq \mathfrak{g}(h, \sum_{i \in A} a_i)$, while if $0 \in A$ then $\mathfrak{g}_{-\phi(J)} \subseteq \mathfrak{g}(h, \sum_{i \in A} a_i)$. In any case, $\dim \mathfrak{g}(h, \sum_{i \in A} a_i) \geq |{}_A \Psi_B|$.

\subsection{} As an example of how this can be applied, consider the root system of type $E_6$, and suppose we are looking for balanced toral elements in characteristic $p=7$. The Richardson orbit with label $2A_2$ has dimension $48$, which is divisible by $p-1$. It can be induced only by the $0$-orbit on a Levi subalgebra whose root system is of type $D_4$. There is a unique way of embedding a centraliser of type $D_4$ in the Dynkin diagram of $E_6$; therefore, if a balanced toral element $h$ exists for this orbit, its Kac coordinates satisfy $a_2=a_3=a_4=a_5=0$, while $a_0, a_1, a_6$ are all strictly positive and verify $a_0= 7 - a_1 - a_6$. Moreover, eigenspaces for $h$ corresponding to nonzero eigenvalues must have dimension $8$.

Since $|{}_{0} \Psi_{1,6}|=|{}_{1} \Psi_{0,6}|=|{}_{6} \Psi_{0,1}|=6$, necessarily $a_0, a_1, a_6$ are three \emph{distinct} strictly positive integers. Up to symmetries of the extended Dynkin diagram, the only option (see Remark \ref{LastRem} below) is
\begin{displaymath}
 a_1=1 \qquad a_6=2  \qquad a_0=4.
 \end{displaymath}

\noindent Direct computation shows that this is indeed a $d$-balanced toral element for every positive integer $d$ dividing $8$.

\subsection{} \label{LastRem}
When checking that an element $h$ corresponding to a choice of Kac coordinates is balanced, one can actually consider its coordinates up to any symmetry of the extended Dynkin diagram, and not only up to the action of an element of $\Omega$. The bigger group $Aut (\mathfrak{g})$ is the union of laterals $\sigma G$  where $\sigma$ is a diagram automorphism of the Dynkin diagram, and the fact that Kac coordinates coincide up to a symmetry not in $G$ only means that the corresponding toral elements are conjugated by some element in $\sigma G$. This, despite giving a possibly different $G$-conjugacy class, does not change the dimension of eigenspaces.

For exceptional Lie algebras, only in type $E_6$ the full automorphisms group $Aut (\mathfrak{g})$ is not connected, but consists of two distinct connected components.

\end{document}